\documentclass[12 pt,a4paper]{amsart} 
\usepackage{amsfonts,amssymb,amscd,amsmath,enumerate,verbatim,calc} 
\usepackage{float}
\usepackage{amsthm}
\newtheorem{theorem}{Theorem}[section]
\newtheorem{definition}{Definition}[section]

\newtheorem{lemma}[theorem]{Lemma}
\newtheorem{corollary}[theorem]{Corollary}
\newtheorem{proposition}[theorem]{Proposition}
\newtheorem{remark}[theorem]{Remark}
\usepackage{tikz-cd}
\usepackage{mathtools}
\usepackage{xfrac}
\usepackage[all,cmtip]{xy}
\newtheorem{example}[theorem]{Example}
\numberwithin{equation}{section}

\usepackage{amsfonts}

\newcommand{\im}{\rm Im}

\renewcommand{\ker}{\textnormal{ker}}
\renewcommand{\im}{\textnormal{Im}}

\begin{document}
\title{Tensor square and isoclinic extensions of multiplicative Lie algebras }
\author{Dev Karan Singh $^{1}$, Amit Kumar$^{2}$, Sumit Kumar Upadhyay$^{3}$ AND Shiv Datt Kumar$^{4}$\vspace{.4cm}\\
{$^{1, 4}$Department of Mathematics, \\ Motilal Nehru National Institute of Technology Allahabad \\ Prayagraj (UP) - 211004, India}\\
		{$^{2, 3}$Department of Applied Sciences,\\ Indian Institute of Information Technology Allahabad\\Prayagraj, U. P., India} }
	
	\thanks{$^1$devkaransingh1811@gmail.com, $^2$amitiiit007@gmail.com, $^3$upadhyaysumit365@gmail.com, $^4$sdt@mnnit.ac.in.  }
	

\begin{abstract}
 In this paper, we discuss the capable and isoclinic properties of the tensor square in the context of multiplicative Lie algebras. We also developed the concept of isoclinic extensions and proved several results for multiplicative Lie algebras. Consequently, we demonstrate that covers of a multiplicative Lie algebra are mutually isoclinic.
\end{abstract}
	\maketitle
	
\section{Introduction}
The theory of multiplicative Lie algebra has evolved significantly over the past few years. In this development, Pandey and Upadhyay \cite{MS} introduced the concept of isoclinism in the setting of multiplicative Lie algebras, which can be helpful in the classification of multiplicative Lie algebra structures in a given group. The notion of isoclinism was established by P. Hall for the classification of p-groups \cite{P.Hall}. In 1994, Moneyhun \cite{Moneyhun} extended the notion of isoclinism to Lie algebras. Extending the idea of isoclinism in Lie algebras, the authors \cite{Mohammadzadeh} defined isoclinic extension and proved that the covers of a finite dimensional Lie algebra are mutually isoclinic. So, exploring the theory of isoclinic extensions to multiplicative Lie algebras is a matter of interest. In their work \cite{AMS}, Kumar et al. recently introduced the concept of capable multiplicative Lie algebras. The notion of a capable group, first introduced by Baer \cite{Baer}, involves the systematic investigation of conditions under which a group can serve as the group of inner automorphisms of another group. Also, significant developments have been made in studying capable Lie algebras. This further prompts us to explore the concept of capability for multiplicative Lie algebras.
A non-abelian tensor product of multiplicative Lie algebras was introduced by Donadze et al. in their work \cite{GNM}.  The authors proved that this notion recovers the existing concept of non-abelian tensor product of groups and Lie algebras.  In \cite{DASS}, several properties of Lie nilpotency and Lie solvability are discussed in the context of the non-abelian tensor product of multiplicative Lie algebras. Also,  upper bounds for the Lie nilpotency class and the Lie solvability length of the quotient $\frac{G\otimes H}{I}$ are established for some ideal $I$ of $G\otimes H$. This motivates us to explore the concepts of capability and isoclinism in the tensor product of multiplicative Lie algebras.
 
In this paper, we review some foundational concepts of multiplicative Lie algebras. Then, in section 3, we establish some results to examine the capability and isoclinism of the non-abelian tensor square of multiplicative Lie algebras. In section 4, we extend the concept of isoclinism to central extensions of multiplicative Lie algebras. We introduce the notion of isoclinic extensions for multiplicative Lie algebras and establish several fundamental results in this context. Specifically, we establish equivalent conditions for two central extensions of multiplicative Lie algebras to be isoclinic. We also demonstrate that the concept of isoclinism between two central extensions is equivalent to isomorphism under certain conditions. We see that every central extension is isoclinic to a stem extension. Finally, we explore the relationship between isoclinic extensions and the Schur multiplier of multiplicative Lie algebras. As an application, we prove that all stem covers are mutually isoclinic, and so covers of a multiplicative Lie algebra are mutually isoclinic.

\section{Preliminaries}
In this section, we recall some basic definitions concerning multiplicative Lie algebras, which will be used throughout this paper.
\begin{definition}\cite{GJ}
	A multiplicative Lie algebra is a triple $ (G,\cdot,\star), $ where $ (G,\cdot) $ is a group together with a binary operation $ \star $ on $G$ such that the following identities hold: 
	\begin{enumerate}
		\item $ x\star x=1 $  
		\item $ x\star(yz)=(x\star y){^y(x\star z)} $ 
		\item $ (xy)\star z= {^x(y\star z)} (x\star z) $ 
		\item $ ((x\star y)\star {^yz})((y\star z)\star{^zx})((z\star x)\star{^xy})=1 $ 
		\item $ ^z(x\star y)=(^zx\star {^zy})$ 
	\end{enumerate}
for all $x,y,z\in G$, where $^xy$ denotes $xyx^{-1}$. We say $ \star $ is a multiplicative Lie algebra structure on the group $G$.
\end{definition}
%%%%%%%%%%%%%%%%%%%%%%%%%%%%%%%%%%%%%%%%%%%%%%%%%%%%%%%%%%%%%%%
\begin{definition}\label{D1}
Let $(G,\cdot,\star)$ be a multiplicative Lie algebra. Then
\begin{enumerate}
\item A subgroup $H$ of $G$ is said to be a subalgebra of G if $x\star y \in H$ for all $x,y \in H$.

\item A subalgebra $H$ of $G$ is said to be an ideal of $G$ if it is a normal subgroup of $G$ and $x\star y \in H$ for all $x \in G$ and $y \in H$. The ideal generated by $\{ a\star b ~\mid ~ a, b \in G\} $ is denoted by $G\star G$.

\item Let $(G',\circ , \star')$ be another multiplicative Lie algebra. A group homomorphism $\psi: G \to G'$ is called a multiplicative Lie algebra homomorphism if $\psi(x\star y) =\psi(x) \star ' \psi(y)$ for all $x, y \in G$.

\item The ideal	$ LZ(G) = \{x \in G \mid  x \star  y = 1 $ for all $y\in G\}$  is called the Lie center of $G.$ 

\item The group center	$ Z(G) = \{x \in G \mid [x,y]=1 $ for all $y\in G\}$  is an ideal  of $G.$
\end{enumerate}
\end{definition}
%%%%%%%%%%%%%%%%%%%%%%%%%%%%%%%%%%%%%%%%%%%%%%%%%%%%%%%%%%%%%%%%%%%%%%
\begin{remark}
\begin{enumerate}
    \item We denote the ideal $ LZ(G)\cap Z(G) $ of $G$ by $ \mathcal{Z}(G)$.

    \item The ideal $(G\star G)[G,G]$ of $G$ is denoted by  ${^M[G,G]}$.
\end{enumerate}
	 
\end{remark}
%%%%%%%%%%%%%%%%%%%%%%%%%%%%%%%%%%%%%%%%%%%%%%%%%%%%%%%%%%%%%%%%%%%%%%%%%%%

\begin{definition} \cite{GNM}
	Let $G$ and $H$ be two multiplicative Lie algebras. By an action of $G$ on $H$ we mean an underlying group action of  $G$ on  $H$ and $H$ on  $G$, together with a map $G \times H \to H, (x, y) \to \langle x, y \rangle,$ satisfying the
	following conditions:
\end{definition}
\hspace{2.5cm}	$ \langle x, yy' \rangle  = \langle x, y\rangle \langle ^yx, {^yy'}\rangle, $

\hspace{2.5cm}   $\langle xx', y \rangle  = \langle ^xx', {^xy} \rangle \langle x, y\rangle,$ 

\hspace{2.5cm}	$\langle (x \star x'), {^{x'}y} \rangle \langle ^yx, \langle x', y\rangle \rangle ^{-1} \langle ^xx', {\langle x, y \rangle} ^{-1}\rangle^{-1} = 1, $

\hspace{2.5cm}	$\langle ^{y'}x, (y\star y')\rangle \langle  {\langle y, x\rangle}^{-1}, {^yy'}\rangle^{-1} \langle  {\langle y', x\rangle}, {^xy}\rangle ^{-1} = 1,$  

where $x, x' \in G, y, y'\in H, {^xx'} = xx'x^{-1}, {^yy'} = yy'y^{-1}$ and  ${^xy}$ and ${^yx}$ denote group actions of $x$ on $y$ and $y$ on $x$, respectively.

%%%%%%%%%%%%%%%%%%%%%%%%%%%%%%%%%%%%%%%%%%%%%%%%%%%%%%%%%

\begin{definition} \cite{GNM} 
	Let $G$ and $H$ be two multiplicative Lie algebras acting on each other. Then the non-abelian tensor product $G \otimes H$ is the multiplicative Lie algebra generated by the symbols $x \otimes y$  subject to the following relations:
\end{definition}
\begin{enumerate}
	\item 	$x \otimes (yy') = (x \otimes y) {(^yx \otimes ^yy')}$ \label{C11} 
	
	\item   $(xx') \otimes y = {(^xx' \otimes ^xy)} (x \otimes y)$ \label{C7}
	
	\item $((x \star x') \otimes {^{x'}y}) (^yx \otimes \langle x', y\rangle)^{-1} ({^xx'} \otimes {\langle x, y \rangle} ^{-1})^{-1} = 1$ \label{C8}
	
	\item $({^{y'}x}  \otimes (y\star y')) ( {\langle y, x\rangle}^{-1}  \otimes {^yy'})^{-1} ( {\langle y', x\rangle}  \otimes {^xy})^{-1} = 1$ \label{C9} 
	
	\item $ ((x \otimes y) \star (x' \otimes y')) = {\langle y, x \rangle}^{-1} \otimes {\langle x', y' \rangle}. $ \label{C10}
	
\end{enumerate}

for all $x, x' \in G \ \text{and} \ y, y' \in H.$

%%%%%%%%%%%%%%%%%%%%%%%%%%%%%%%%%%%%%%%%%%%%%%%%%%%%%%%%

\begin{definition}\cite{AMS,RLS}\label{D2}
\begin{enumerate} 
    \item A short exact sequence 
   $$\mathcal{E}\equiv  \xymatrix{1\ar[r] & H\ar[r]^{\alpha} & G\ar[r]^{\beta} & K\ar[r] & 1}$$ of multiplicative Lie algebras is called an extension of $H$ by $K$.

    \item An extension $\mathcal{E}$ of a multiplicative Lie algebra $H$ by $K$ is called a central extension if $H\subseteq \mathcal{Z}(G)$.

    \item 	A central extension $$\mathcal{C}\equiv  \xymatrix{1\ar[r] & H\ar[r]^{i} & G\ar[r]^{\beta} & K\ar[r] & 1}$$ of multiplicative Lie algebras is called an stem extension if $H \subseteq {^M[G,G]}.$ If, in addition, $H\cong \tilde M(K)$ (The Schur multiplier of multiplicative Lie algebra $K$), the above extension is called a stem cover. In this case, $G$ is said to be a cover of $K$.
    
    \item Let $EXT$ denote the category whose objects are short exact sequences of multiplicative Lie algebras. A morphism from a short exact sequence $\mathcal{E}$ to a short exact sequence  $$\mathcal{E'}\equiv  \xymatrix{1\ar[r] & H'\ar[r]^{\alpha'} & G'\ar[r]^{\beta'} & K'\ar[r] & 1}$$
    is a triple $(\lambda,\mu,\nu)$, where $\lambda$ is a homomorphism from $H$ to $H'$, $\mu$ is a  homomorphism from $G$ to $G'$, and $\nu$ is a homomorphism from $K$ to $K'$ such that the relevant diagram is commutative.
\end{enumerate}

\end{definition}

%%%%%%%%%%%%%%%%%%%%%%%%%%%%%%%%%%%%%%%%%%%%%%%%%%%%%%%%%%%%%%%%%%%%%%%

    \begin{definition}\cite{MS}\label{isoclinism_def}
Two  multiplicative Lie algebras $G_1$ and $G_2$ are said  to be isoclinic (written as $G_1 \sim_{ml} G_2$) if  there exist multiplicative Lie algebra isomorphisms $\lambda:\frac{G_1}{\mathcal{Z}(G_1)} \to \frac{G_2}{\mathcal{Z}(G_2)}$ and $\mu:\ {^M[G_1, G_1]}\to \ {^M[G_2, G_2]}$ such that the following diagram

\[
\xymatrix{
^M[G_1, G_1]\ar[d]^{\mu} & \frac{G_1}{\mathcal{Z}(G_1)}\times \frac{G_1}{\mathcal{Z}(G_1)}\ar[l]^{\phi_{c}~~~~~}  \ar[r]_{~~~~\phi_{s}} \ar[d]_{\lambda\times \lambda} &^M[G_1, G_1]  \ar[d]_{\mu} \\
^M[G_2, G_2] & \frac{G_2}{\mathcal{Z}(G_2)}\times \frac{G_2}{\mathcal{Z}(G_2)}\ar[l]_{\psi_{c}~~~~~~} \ar[r]^{~~~~\psi_{s}} & ^M[G_2,G_2]
}
\]
is commutative. The pair $(\lambda,\mu)$ is called an isoclinism between the multiplicative Lie algebras $G_1$ and $G_2$.
\end{definition}
%%%%%%%%%%%%%%%%%%%%%%%%%%%%%%%%%%%%%%%%%%%%%%%%%%%%
\section{Some results on tensor square of multiplicative Lie algebra}
We start the section by defining the tensor square of a multiplicative Lie algebra. 

Let $G$ be a multiplicative Lie algebra. Then, the action of the underlying group on itself is defined by conjugation with $\langle, \rangle:= \star$ is an action of multiplicative Lie algebra $G$ on itself. In this case, 
 we call $G \otimes G$ the non-abelian tensor square (sometimes the tensor square) of $G$. 
 
  Now, if we have a Lie capable multiplicative Lie algebra $G$ (see \cite{AMS}),  is $G\otimes G$  also Lie capable? If not, then under which conditions is Lie capable? Also, if $G_1$ and $G_2$ are isoclinic multiplicative Lie algebras, then what can we say about $G_1\otimes G_1$ and $G_2\otimes G_2$ in the sense of isoclinism? To address these questions, we begin by establishing the following lemmas.
\begin{lemma}\label{Trivial}
    Let $I$ be an ideal of a multiplicative Lie algebra $G$. Then 
    \begin{enumerate}
    \item If $I\cap [G,G]=1$, then ${{Z}\big(\frac{G}{I}\big)} =\frac{{Z}(G)}{I}$.
    
    \vspace{.15cm}
    \item If $I\cap (G\star G)=1$, then ${{LZ}\big(\frac{G}{I}\big)} =\frac{{LZ}(G)}{I}$.
    
    \vspace{.15cm}
    \item If $I\cap {^M[G,G]}=1$, then ${\mathcal{Z}\big(\frac{G}{I}\big)} =\frac{\mathcal{Z}(G)}{I}$.
    
    \end{enumerate}
   
\end{lemma}
%%%%%%%%%%%%%%%%%%%%%%%%%%%%%%%%%%%%%%%%%%%%%%%%%%%%%%%

\begin{proof}
\textit{(1)} Follows directly from group-theoretic proof.

\textit{(2)}	Let $g\in G$ and $h\in I$. Then $(g\star h) \in I\cap (G\star G)=1$. This implies $g\star h =1$. Thus $h\in {LZ}(G)$. Therefore $I\subseteq {LZ}(G)$ and so $I$ is also an ideal of ${LZ}(G)$. Now, let $gI\in G/I$. Then
	\begin{align*}
	gI \in {LZ}(G/I)	\iff& gI\star g_1I=I~  \text{for all} ~ g_1I\in G/I\\
		\iff& g\star g_1\in I \\
		\iff& (g\star g_1) \in I\cap (G\star G) \\
		\iff& g \in {LZ}(G) \\
			\iff& gI \in {LZ}(G)/I.
	\end{align*}
 Similarly, one can prove \textit{(3)}.
 \end{proof}
 %%%%%%%%%%%%%%%%%%%%%%%%%%%%%%%%%%%%%%%%%%%%%%%%%%%%%%%%%
 \begin{lemma}\label{sixth_lemma}
     Let  $I$ be an ideal of a multiplicative Lie algebra $G$. Then $$\frac{G}{I}\otimes \frac{G}{I} \cong \frac{G\otimes G}{(I\otimes G)(G\otimes I)}.$$
 \end{lemma}
%%%%%%%%%%%%%%%%%%%%%%%%%%%%%%%%%%%%%%%%%%%%%%%%%%%%%%
\begin{proof}
We define a map $\phi: G \otimes G \to \frac{G}{I}\otimes \frac{G}{I}$ by $\phi(g_1\otimes g_2)=g_1I\otimes g_2I$. Then, $\phi$ is well-defined. To prove that $\phi$ is a homomorphism, we need to show the following identities:
	\begin{gather*}
	\phi(g_1g_2\otimes g_1')= \phi(^{g_1}g_2\otimes {^{g_1}}g_1') \phi(g_1\otimes g_1')\\
  \phi(g_1\otimes g_1'g_2')= \phi(g_1\otimes g_1') \phi(^{g_1'}g_1\otimes {^{g_1'}}g_2')\\
   \phi((g_1\star g_2)\otimes {^{g_2}g_1'})\phi({^{g_1'}g_1} \otimes (g_2\star g_1'))^{-1} \phi({^{g_1}g_2} \otimes (g_1'\star g_1))^{-1} =1\\
 \phi(^{g_2'}g_1 \otimes(g_1'\star g_2')) \phi((g_1\star g_1')  \otimes {^{g_1'}g_2'})^{-1} \phi((g_2'\star g_1) \otimes {^{g_1}g_1'})^{-1} =1\\
    \phi((g_1\otimes g_2) \star (g_1' \otimes g_2')) = \phi((g_1\star g_2)  \otimes (g_1'\star g_2'))
	\end{gather*}
	We will prove the first and third identities using \cite[Definition 3.2]{GNM}; the rest are similar and easy to prove.
	\begin{multline*}
			\phi(g_1g_2\otimes g_1')  = g_1g_2I \otimes g_1'I= g_1Ig_2I \otimes g_1'I
		= (^{g_1I}g_2I\otimes {^{g_1I}g_1'I})(g_1I \otimes g_1'I) \\
		= (^{g_1}g_2I\otimes {^{g_1}g_1'I})(g_1I \otimes g_1'I) 
		 = \phi(^{g_1}g_2\otimes {^{g_1}}g_1') \phi(g_1\otimes g_1').
	\end{multline*}
 Next,
   \begin{multline*}
   		\phi((g_1\star g_2)\otimes {^{g_2}g_1'})\phi({^{g_1'}g_1} \otimes (g_2\star g_1'))^{-1} \phi({^{g_1}g_2} \otimes (g_1'\star g_1))^{-1} \\ = 	((g_1\star g_2)I\otimes {^{g_2}g_1'}I)({^{g_1'}g_1}I \otimes (g_2\star g_1')I)^{-1} ({^{g_1}g_2}I \otimes (g_1'\star g_1)I)^{-1} \\ = 	((g_1I\star g_2I)\otimes {^{g_2I}g_1'I})({^{g_1'I}g_1I} \otimes (g_2I\star g_1'I))^{-1} ({^{g_1I}g_2I} \otimes (g_1'I\star g_1I))^{-1} = 1.
   \end{multline*}

Clearly, $(I\otimes G)(G\otimes I) \subseteq \ker \phi$.  Without loss of generality, let $(g_1\otimes g_2) \in \ker \phi$, $\phi(g_1\otimes g_2)=g_1I \otimes g_2I=1$. Suppose  $g_1h\otimes g_2k=1$ for some $h,k\in I$. It follows that $(g_1\otimes g_2)=(^{g_1}h\otimes {^{g_1}g_2})^{-1}(^{g_2}{(g_1h)}\otimes {^{g_2}k})^{-1}\in (I\otimes G)(G\otimes I)$. Thus $\ker \phi \subseteq (I\otimes G)(G\otimes I)$. Hence $\phi$ induces an isomorphism $\tilde \phi:\frac{G \otimes G}{(I\otimes G)(G\otimes I)} \to \frac{G}{I}\otimes \frac{G}{I}$.
\end{proof}

 %%%%%%%%%%%%%%%%%%%%%%%%%%%%%%%%%%%%%%%%%%%%%%%%%%%%%%%
 \begin{lemma}\label{contained}
     For a multiplicative Lie algebra $G$, the following conditions hold:
    \begin{enumerate}
    \item $({Z}(G)\otimes G)(G\otimes{Z}(G))\subseteq {Z}(G\otimes G)$. 

    \item $({LZ}(G)\otimes G)(G\otimes{LZ}(G))\subseteq {LZ}(G\otimes G)$.
    
    \item  $(\mathcal{Z}(G)\otimes G)(G\otimes \mathcal{Z}(G))\subseteq \mathcal{Z}(G\otimes G)$. 
    \end{enumerate}

 \end{lemma}
%%%%%%%%%%%%%%%%%%%%%%%%%%%%%%%%%%%%%%%%%%%%%%%%%%%%%%%%%%%%%%

\begin{proof}
\textit{(1)} Follows from the identity $[g\otimes h, g'\otimes h']=[g,h]\otimes [g',h']$ for all $g, g', h, h'\in G$.

\textit{(2)} Without any loss, we suppose $ (z\otimes g)\in {LZ}(G)\otimes G $ and $(g_1\otimes g_2) \in G\otimes G$, where $z\in {LZ}(G)$ and $g,g_1,g_2 \in G$. Then  $ (z\otimes g)\star (g_1\otimes g_2)= (z\star g)\otimes (g_1\star g_2)=1 $ . This implies  $ (z\otimes g)\in {LZ}(G\otimes G)$. Hence, the result follows.

Similarly, one can easily prove \textit{(3)}.
\end{proof}
%%%%%%%%%%%%%%%%%%%%%%%%%%%%%%%%%%%%%%%%%%%%%%%%%%%%%%%%%
 \begin{proposition}\label{quotient}
   Let $G$ be a Lie capable multiplicative Lie algebra. Then $\frac{G\otimes G}{I}$ is also Lie capable for some ideal $I$.
\end{proposition}
%%%%%%%%%%%%%%%%%%%%%%%%%%%%%%%%%%%%%%%%%%%%%%%%%%%%%%%%
\begin{proof}
    Since $G$ is Lie capable, $G\cong \frac{E}{\mathcal{Z}(E)}$ for some multiplicative Lie algebra $E$. Thus, we have $G\otimes G \cong \frac{E}{\mathcal{Z}(E)}\otimes \frac{E}{\mathcal{Z}(E)} \cong \frac{E\otimes E}{(\mathcal{Z}(E)\otimes E)(E\otimes \mathcal{Z}(E))}$ by Lemma \ref{sixth_lemma}. Also, $\frac{E\otimes E}{\mathcal{Z}(E\otimes E)} \cong \frac{\frac{E\otimes E}{(\mathcal{Z}(E)\otimes E)(E\otimes \mathcal{Z}(E))}}{\frac{\mathcal{Z}(E\otimes E)}{(\mathcal{Z}(E)\otimes E)(E\otimes \mathcal{Z}(E))}}\cong \frac{G\otimes G}{I}$, where $I=\frac{\mathcal{Z}(E\otimes E)}{(\mathcal{Z}(E)\otimes E)(E\otimes \mathcal{Z}(E))}$ is an ideal of $G\otimes G$. Hence, the result follows.

\end{proof}
%%%%%%%%%%%%%%%%%%%%%%%%%%%%%%%%%%%%%%%%%%%%%%%%%%%%%%%%%%%%%
\begin{corollary}
  Let $G$ be a Lie capable multiplicative Lie algebra, i.e., $G \cong \frac{E}{\mathcal{Z}(E)}$ for some multiplicative Lie algebra $E$. If ${(\mathcal{Z}(E)\otimes E)(E\otimes \mathcal{Z}(E))}\cap {^M[E\otimes E,E\otimes E]}=1$, then $\frac{G\otimes G}{\mathcal{Z}(G\otimes G)}$ is also Lie capable.
\end{corollary}
%%%%%%%%%%%%%%%%%%%%%%%%%%%%%%%%%%%%%%%%
\begin{proof}
   By Lemma \ref{sixth_lemma}, we have $G\otimes G  \cong \frac{E\otimes E}{(\mathcal{Z}(E)\otimes E)(E\otimes \mathcal{Z}(E))}$. Moreover, ${\mathcal{Z}(G\otimes G)}\cong {\mathcal{Z}\big(\frac{E\otimes E}{(\mathcal{Z}(E)\otimes E)(E\otimes \mathcal{Z}(E))}\big)} = \frac{\mathcal{Z}(E\otimes E)}{(\mathcal{Z}(E)\otimes E)(E\otimes \mathcal{Z}(E))} $ by Lemma \ref{Trivial}. Hence, we have $\frac{\frac{E\otimes E}{(\mathcal{Z}(E)\otimes E)(E\otimes \mathcal{Z}(E))}}{\mathcal{Z}\big(\frac{E\otimes E}{(\mathcal{Z}(E)\otimes E)(E\otimes \mathcal{Z}(E))}\big)}\cong \frac{G\otimes G}{\mathcal{Z}(G\otimes G)}$ and we are done.
\end{proof}

%%%%%%%%%%%%%%%%%%%%%%%%%%%%%%%%%%%%%%%%%%%%%%%%%%%%%%%%%%%%%%%%%
\begin{proposition}
    Let $G$ be a Lie capable multiplicative Lie algebra, i.e., $G \cong \frac{E}{\mathcal{Z}(E)}$ for some multiplicative Lie algebra $E$. If $\mathcal{Z}(E\otimes E) = {(\mathcal{Z}(E)\otimes E)(E\otimes \mathcal{Z}(E))}$, then $G\otimes G$ is also Lie capable.
\end{proposition}
%%%%%%%%%%%%%%%%%%%%%%%%%%%%%%%%%%%%%%%%%%%%%%%%%%%%%%%%%%%
\begin{proof}
  Since  $G\otimes G \cong \frac{E}{\mathcal{Z}(E)}\otimes \frac{E}{\mathcal{Z}(E)} \cong \frac{E\otimes E}{(\mathcal{Z}(E)\otimes E)(E\otimes \mathcal{Z}(E))}$ and $\mathcal{Z}(E\otimes E) = (E\otimes \mathcal{Z}(E))(\mathcal{Z}(E)\otimes E)$, we have $G\otimes G \cong \frac{E\otimes E}{\mathcal{Z}(E\otimes E)}$.
\end{proof}
%%%%%%%%%%%%%%%%%%%%%%%%%%%%%%%%%%%%%%%%%%%%%%%%%%%%%%
Next, we provide some examples for a multiplicative Lie algebra $G$ for which the condition $\mathcal{Z}(G\otimes G) = (G\otimes \mathcal{Z}(G))(\mathcal{Z}(G)\otimes G)$ holds. 

\begin{example}
\begin{enumerate}
    \item If we consider a finite cyclic group $C_n$ of order $n$,  there is only trivial multiplicative Lie algebra structure on $C_n$. In this case, $\mathcal{Z}(C_n\otimes C_n) = (C_n\otimes \mathcal{Z}(C_n))(\mathcal{Z}(C_n)\otimes C_n)$.

    \item  Let us look at the Klein four group $V_4 = \langle a, b \mid a^2 = b^2 = 1, ab = ba \rangle$ with trivial multiplicative Lie algebra structure. It is easy to see that $\mathcal{Z}(V_4\otimes V_4) = V_4\otimes \mathcal{Z}(V_4)$ using \cite[Proposition 3.4]{GNM}. Hence, the required condition holds.

\end{enumerate}
  \end{example}
%%%%%%%%%%%%%%%%%%%%%%%%%%%%%%%%%%%%%%%%%%%%%%%%%%%%%%%%%%%%%

\begin{proposition}
	Let $G_1$ and $G_2$ be multiplicative Lie algebras such that $G_1 \sim_{ml} G_2$ via $(\lambda,\mu)$. If $\mathcal{Z}(G_i\otimes G_i) = {(\mathcal{Z}(G_i)\otimes G_i)(G_i\otimes \mathcal{Z}(G_i))}$ for $i=1,2$. Then $G_1\otimes G_1 \sim_{ml} G_2\otimes G_2$.
\end{proposition}
%%%%%%%%%%%%%%%%%%%%%%%%%%%%%%%%%%%%%%%%%%%%%%%%%%%%%%%%%%%%%%%%%%

\begin{proof}
	Given $\mathcal{Z}(G_i\otimes G_i) = {(\mathcal{Z}(G_i)\otimes G_i)(G_i\otimes \mathcal{Z}(G_i))}$. Therefore, by Lemma \ref{sixth_lemma}, we have  $\frac{G_i \otimes G_i}{\mathcal{Z}(G_i \otimes G_i)}\cong \frac{G_i}{\mathcal{Z}(G_i)}\otimes \frac{G_i}{\mathcal{Z}(G_i)}$.  Also $ \frac{G_1}{\mathcal{Z}(G_1)}\cong \frac{G_2}{\mathcal{Z}(G_2)}$ via $\lambda$ implies that  $\frac{G_1 \otimes G_1}{\mathcal{Z}(G_1 \otimes G_1)} \cong \frac{G_2 \otimes G_2}{\mathcal{Z}(G_2 \otimes G_2)}$  via $\bar \lambda$ (say). Next, we know  $\mu: {^M[G_1,G_1]} \to {^M[G_2,G_2]}$ is an isomorphism. By Definition \ref{isoclinism_def}, $\mu|_{(G_1 \star G_1)}: (G_1 \star G_1) \to (G_2 \star G_2)$ and $\mu|_{[G_1,G_1]} : [G_1,G_1] \to [G_2,G_2]$ are also isomorphisms. It is easy to see that the map $\bar\mu: {^M[G_1\otimes G_1, G_1\otimes G_1]}\to {^M[G_2\otimes G_2, G_2\otimes G_2]}$ defined by $$\bar{\mu}\left([(g_1\otimes g_2),(g_3\otimes g_4)]\right)= \bar{\mu}([g_1, g_2]\otimes [g_3,g_4])= \mu[g_1,g_2]\otimes \mu[g_3,g_4]$$ and  $$\bar{\mu}\left((g_1'\otimes g_2')\star(g_3'\otimes g_4')\right)= \bar{\mu}((g_1'\star g_2')\otimes (g_3'\star g_4'))= \mu(g_1'\star g_2')\otimes \mu(g_3'\star g_4')$$  is an isomorphism and the following diagram is commutative:
      \[
    \xymatrix{
    	^M[G_1\otimes G_1, G_1\otimes G_1]\ar[d]^{\tilde{\mu}} & \frac{G_1\otimes G_1}{\mathcal{Z}(G_1\otimes G_1)}\times \frac{G_1\otimes G_1}{\mathcal{Z}(G_1\otimes G_1)}\ar[l]^{\tilde{\phi_{c}}}  \ar[r]_{\tilde{\phi_{s}}} \ar[d]_{\tilde{\lambda}\times \tilde{\lambda}} &^M[G_1\otimes G_1, G_1\otimes G_1]  \ar[d]_{\tilde{\mu}} \\
    	^M[G_2\otimes G_2, G_2\otimes G_2] & \frac{G_2\otimes G_2}{\mathcal{Z}(G_2\otimes G_2)}\times \frac{G_2\otimes G_2}{\mathcal{Z}(G_2\otimes G_2)}\ar[l]_{\tilde{\psi_{c}}} \ar[r]^{\tilde{\psi_{s}}} & ^M[G_2\otimes G_2,G_2\otimes G_2].
    }\]
\end{proof}

%%%%%%%%%%%%%%%%%%%%%%%%%%%%%%%%%%%%%%%%%%%%%%%%%%%%%%%%

\section{Isoclinic extensions}

In the present section, extending the notion of isoclinism between the multiplicative Lie algebras, we introduce the concept of isoclinic extensions and establish several key results to develop the theory of isoclinic extensions in multiplicative Lie algebras. 

\begin{definition}\label{Isoclinism_of_extensions}
		Let $\mathcal{E}_i \equiv \xymatrix{1\ar[r] & H_i\ar[r]^{i} & G_i \ar[r]^{\beta_i} & K_i\ar[r] & 1}$, $i=1,2$, be two central extensions of multiplicative Lie algebras. Then $\mathcal{E}_1$ and $\mathcal{E}_2$ are said to isoclinic and denoted by $\mathcal{E}_1 \sim_{ml} \mathcal{E}_2$, if there exist multiplicative Lie algebra isomorphisms $\lambda: K_1 \to K_2$ and $\mu:{^M[G_1,G_1]} \to  {^M[G_2,G_2]}$ such that the following diagram is commutative:
		\[
		\xymatrix{
			^M[G_1,G_1] \ar[d]^{\mu  } & \ar[l]_{\delta_1^c} K_1 \times K_1 \ar[r]^{\delta_1^s} \ar[d]^{\lambda  \times \lambda } & ^M[G_1,G_1] \ar[d]_{\mu } \\
			^M[G_2,G_2] & \ar[l]_{\delta_2^c}  K_2 \times K_2 \ar[r]^{\delta_2^s} & ^M[G_2,G_2]
		}
		\]
  where the maps $\delta_i^c$ and $\delta_i^s$,  given by $\delta_i^c(k_{i1},k_{i2})=[g_{i1},g_{i2}]$ and $\delta_i^s(k_{i1}',k_{i2}')=(g_{i1}'\star g_{i2}')$ with $\beta_i(g_{ij})=k_{ij}$ and $\beta_i(g_{ij}')=k_{ij}'$, where $ i,j \in \{1,2\} $, are well defined. The pair $(\lambda,\mu )$ is said to be an isoclinism between the central extensions $\mathcal{E}_1$ and $\mathcal{E}_2$.

  In particular, the concept of isoclinism between  $\mathcal{E}_1$ and $\mathcal{E}_2$ with  $H_1=\mathcal{Z}(G_1)$ and $H_2=\mathcal{Z}(G_2)$ coincides with the concept of  isoclinism between  $G_1$ to $G_2$ .

\end{definition}

%%%%%%%%%%%%%%%%%%%%%%%%%%%%%%%%%%%%%%%%%%%%%%%%%%%%%%%%%%%%%%%%%%%%%%%%%%%%%%%%%%%
In the following lemma, we explore certain properties of isoclinic extensions.
\begin{lemma}\label{first_lemma}
		Let $(\lambda,\mu )$ be an isoclinism between the central extensions $\mathcal{E}_1$ and $\mathcal{E}_2$. Then the following holds:
		\begin{enumerate}
			\item $\lambda(\beta_1(x))=\beta_2(\mu(x))$ for all $x\in{ ^M[G_1,G_1]}$. 
   
			\item $\mu({^M[x,g_1]})={^M[\mu(x),g_2]}$ for all $x\in {^M[G_1,G_1]}, g_1\in G_1$ and $g_2\in G_2$ such that $\lambda(\beta_1(g_1))=\beta_2(g_2)$.
   
			\item $\mu(H_1\cap {^M[G_1,G_1]})=H_2\cap{ ^M[G_2,G_2]}$.
		\end{enumerate}
\end{lemma}
%%%%%%%%%%%%%%%%%%%%%%%%%%%%%%%%%%%%%%%%%%%%%%%%%%%%%%%%%%%%%%%%%%%%%%%%%%%%%%%%%%%%%%

\begin{proof}
\textit{(1)~}
Let $x$ be an element of ${^M[G_1,G_1]}$. Without loss of generality, we can assume $x=(g_1\star g_2)[g_3,g_4]$ for some $g_1,g_2,g_3,g_4  \in G_1$. Then 
\begin{equation*}
			\lambda(\beta_1(x))  = \lambda(\beta_1((g_1\star g_2)[g_3,g_4]))= (\lambda\beta_1(g_1) \star \lambda\beta_1(g_2))[\lambda\beta_1(g_3)), \lambda(\beta_1(g_4)].
\end{equation*}
Since $\beta_2$ is surjective, there exist $g_1', g_2'\in G_2$ such that $\beta_2(g_1')=\lambda(\beta_1(g_1))$ and $\beta_2(g_2')=\lambda(\beta_1(g_2))$. 
Therefore Definition \ref{Isoclinism_of_extensions} implies that
			\begin{align*}
				(\delta_2^s (\lambda \times \lambda) )(\beta_1(g_1),\beta_1(g_2)) & = \mu(\delta_1^s(\beta_1(g_1),\beta_1(g_2))) \\
				\delta_2^s (\lambda(\beta_1(g_1)) , \lambda(\beta_1(g_2))) & = \mu(\delta_1^s(\beta_1(g_1),\beta_1(g_2))) \\
				\delta_2^s(\beta_2(g_1'),\beta_2(g_2')) & = \mu (g_1\star g_2)	\\
				g_1' \star g_2'  & = \mu(g_1\star g_2)\\
				\beta_2(g_1' \star g_2') & = \beta_2(\mu(g_1\star g_2)).
			\end{align*}
Thus, we have \[\lambda\beta_1(g_1) \star \lambda\beta_1(g_2)= 	\beta_2(g_1' \star g_2') =  \beta_2(\mu(g_1\star g_2)).\]
Similarly, \[[\lambda\beta_1(g_3)), \lambda(\beta_1(g_4)]= 	 \beta_2(\mu([g_3, g_4])) .\]
Hence 
\begin{equation*}
			\lambda(\beta_1(x)) = \beta_2(\mu(g_1\star g_2))\beta_2(\mu([g_3, g_4])) = \beta_2(\mu((g_1\star g_2)[g_3, g_4])) =\beta_2(\mu(x)).
\end{equation*}

%%%%%%%%%%%%%%%%%%%%%%%%%%%%%%%%%%%%%%%%%%%%%%%%%%%%%%%%%%%%
\textit{(2)~} Let $x\in {^M[G_1,G_1]}$ and $ g_1\in G_1$. Then
			\begin{align*}
				\mu(^M[x,g_1]) & = \mu((x\star g_1)[x,g_1]) \\ 
    & = \mu(x\star g_1)\mu([x,g_1]) \\
    & = \mu(\delta_1^s(\beta_1(x),\beta_1(g_1)))\mu(\delta_1^c(\beta_1(x),\beta_1(g_1))) \\
    & = \delta_2^s(\lambda(\beta_1(x)),\lambda(\beta_1(g_1)))\delta_2^c(\lambda(\beta_1(x)),\lambda(\beta_1(g_1))) \\
	& = \delta_2^s(\beta_2(\mu(x)),\beta_2(g_2))\delta_2^c(\beta_2(\mu(x)),\beta_2(g_2)) \text{~~for some }g_2\in G_2 \\
	& = (\mu(x) \star g_2)[\mu(x), g_2] \\
    & = {^M[\mu(x),g_2]}.
	\end{align*}
			
%%%%%%%%%%%%%%%%%%%%%%%%%%%%%%%%%%%%%%%%%%%%%%%%%%%%%%%%%%%%%%%

\textit{(3)~} Let $x\in H_1\cap~^M[G_1,G_1]$. Then $\beta_2(\mu(x))=\lambda(\beta_1(x))=1$.  This implies $\mu(x)\in \ker(\beta_2)=H_2$. Hence $\mu(x)\in H_2\cap{^M[G_2,G_2]}$. 

Conversely, suppose $y\in H_2\cap{^M[G_2,G_2]}$. Then, there exists $x\in{^M[G_1,G_1]}$ such that $\mu(x)=y$. Therefore $\lambda(\beta_1(x))=\beta_2(\mu(x))=\beta_2(y)=1$ and so $\beta_1(x)=1$. This implies  $x\in H_1\cap {^M[G_1,G_1]}$. Hence $H_2\cap{^M[G_2,G_2]}$ is contained in $\mu(H_1\cap{ ^M[G_1,G_1]})$. This completes the proof.
	\end{proof}
%%%%%%%%%%%%%%%%%%%%%%%%%%%%%%%%%%%%%%%%%%%%%%%%%%%%%%%%%%%%%%%%%
The subsequent proposition demonstrates that the concept of isoclinic extensions expands upon the notion of isoclinism between two multiplicative Lie algebras.
\begin{proposition}\label{first_proposition}
		Let $\mathcal{E}_1 \sim_{ml} \mathcal{E}_2$ via $(\lambda,\mu)$. Then 
		\begin{enumerate}
			\item $G_1$ is isoclinic to   $G_2$ via $(\bar\lambda,\mu)$, where $\bar\lambda: G_1/\mathcal{Z}(G_1)\to G_2/\mathcal{Z}(G_2)$ is an isomorphism induced by $\lambda$.
			\item $H_1=\mathcal{Z}(G_1)$ if and only if $H_2=\mathcal{Z}(G_2)$.
		\end{enumerate}
\end{proposition}

%%%%%%%%%%%%%%%%%%%%%%%%%%%%%%%%%%%%%%%%%%%%%%%%%%%%%%%%%%%%%%%%%%%%%
\begin{proof}
			\textit{(1)~}	Define a map $\phi: G_1 \to G_2/\mathcal{Z}(G_2)$ given by $\phi(g_1)=g_2\mathcal{Z}(G_2)$ such that $\beta_2(g_2)=\lambda(\beta_1(g_1))$.  This map is well-defined. Let $g_2\mathcal{Z}(G_2)\in G_2/\mathcal{Z}(G_2)$. Then there exists $k_1\in K_1$ such that $\lambda(k_1)=\beta_2(g_2)$. Since $\beta_1$ is surjective, there exists $g_1\in G_1$ such that $\lambda(\beta_1(g_1))=\beta_2(g_2)$. This implies $\phi(g_1)=g_2\mathcal{Z}(G_2)$.
            Thus, the map $\phi$ is surjective.
            If we prove  $\ker\phi= \mathcal{Z}(G_1)$, then done. Let $g_1 \in \mathcal{Z}(G_1)$. Then $\phi(g_1)=g_2 \mathcal{Z}(G_2)$ such that $\beta_2(g_2)=\lambda(\beta_1(g_1))$. Take an arbitrary element $g_2'$ in $G_2$. Since $\phi$ is surjective, there exists $g_1'\in G_1$ such that $\phi(g_1')=g_2'\mathcal{Z}(G_2)$. Now
			\begin{align*}
				g_2 \star g_2' & = \delta_2^s (\beta_2(g_2), \beta_2(g_2'))\\
				& = \delta_2^s (\lambda(\beta_1(g_1)),\lambda(\beta_1(g_1')))  \\
				&= \mu(\delta_1^s (\beta_1(g_1),\beta_1(g_1')))\\
				& = \mu(g_1\star g_1') = 1.
			\end{align*}
			Similarly, we can show that $[g_2,g_2']=1$ for all $ g_2'\in G_2$. Thus, we have shown that $g_2\in \mathcal{Z}(G_2) $. Hence $\mathcal{Z}(G_1) \subseteq \ker\phi$.
   
			Next, suppose $g_1\in \ker\phi$, where $\lambda(\beta_1(g_1))=\beta_2(g_2)$ for some $g_2 \in G_2$. This implies $g_2 \in \mathcal{Z}(G_2)$. Take an arbitrary element $g_1'$ in $G_1$.
			\begin{align*}
				\mu(g_1 \star g_1') & = \mu(\delta_1^s (\beta_1(g_1),\beta_1(g_1')))\\
				& = \delta_2^s (\lambda(\beta_1(g_1)),\lambda(\beta_1(g_1')))\\
				& = \delta_2^s (\beta_2(g_2), \beta_2(g_2')) \text{~~for some~~} g_2'\in G_2 \\
                & = g_2 \star g_2' =1.
			\end{align*}
            
			This implies $g_1 \star g_1'=1$ for all $ g_1'\in G_1$. Similarly, we can show that $[g_1,g_1']=1$ for all $ g_1'\in G_1$. Hence $\ker \phi \subseteq \mathcal{Z}(G_1)$. Thus, we have an isomorphism $\bar\lambda : G_1/\mathcal{Z}(G_1) \to G_2/\mathcal{Z}(G_2)$ induced by $\lambda$ such that the following diagram is commutative:
			\[
			\xymatrix{
				^M[G_1,G_1] \ar[d]^{\mu} & \ar[l]_{\bar{\delta_1^c}} \frac{G_1}{\mathcal{Z}(G_1)} \times \frac{G_1}{\mathcal{Z}(G_1)} \ar[r]^{\bar{\delta_1^s}} \ar[d]^{\bar\lambda \times \bar\lambda} & ^M[G_1,G_1] \ar[d]_{\mu} \\
				^M[G_2,G_2] \ar[r]^{\bar{\delta_2^c}} & \frac{G_2}{\mathcal{Z}(G_2)} \times \frac{G_2}{\mathcal{Z}(G_2)} \ar[r]^{\bar{\delta_2^s}} & ^M[G_2,G_2],
			}
			\]
			where $\bar{\delta_i^c}(k_{i1},k_{i2})=[g_{i1},g_{i2}]$ and $\bar{\delta_i^s}(k_{i1}',k_{i2}')=(g_{i1}'\star g_{i2}')$ for $i=1,2$. 
%%%%%%%%%%%%%%%%%%%%%%%%%%%%%%%%%%%%%%%%%%%%%%%%%%%%%%%%%%%%%%%%%%%%%%%

			\textit{(2)~} Let $H_2= \mathcal{Z}(G_2)$ and $h_1 \in \mathcal{Z}(G_1)$. Then there exists $h_2\in \mathcal{Z}(G_2)=H_2=\ker\beta_2$ such that $\lambda(\beta_1(h_1))=\beta_2(h_2)=1$. This implies $h_1\in \ker\beta_1$. Hence $H_1=\mathcal{Z}(G_1)$. Similarly, one can also prove the converse part.
	\end{proof}	

%%%%%%%%%%%%%%%%%%%%%%%%%%%%%%%%%%%%%%%%%%%%%%%%%%%%%%%%%%%%%%%%%%%%%%%%%%%

%%%%%%%%%%%%%%%%%%%%%%%%%%%%%%%%%%%%%%%%%%%%%%%%%%%%%%%%%%%%%%%%%%%%%%%
 An isomorphism in the category $EXT$ is called an equivalence. A morphism $(\lambda,\mu, \nu)$ is an equivalence if and only if $\lambda$ and $\nu$ are isomorphism \cite{RLS}. This motivates the following definition in the sense of isoclinism.
\begin{definition}\label{isoclinism_homomorphism}
		A morphism $(\lambda,\mu, \nu)$ between central extensions $\mathcal{E}_1$ and $\mathcal{E}_2$ is called an isoclinic morphism if $(\nu,\mu')$ is an isoclinism between $\mathcal{E}_1$ and $\mathcal{E}_2$, where $\mu'=\mu|_{^M[G_1,G_1]}$.
\end{definition}
%%%%%%%%%%%%%%%%%%%%%%%%%%%%%%%%%%%%%%%%%%%%%%%%%%%%%%%%%%%%%%%%%

The following theorem gives a necessary and sufficient condition for a morphism to be an isoclinic morphism.
\begin{theorem}\label{first_theorem}
		A morphism $(\lambda,\mu,\nu)$ of central extensions $\mathcal{E}_1$ and $\mathcal{E}_2$ is an isoclinic morphism  if and only if $\nu$ is a multiplicative Lie algebra isomorphism and $\ker\mu \cap {^M[G_1,G_1]}=1$.
\end{theorem}
%%%%%%%%%%%%%%%%%%%%%%%%%%%%%%%%%%%%%%%%%%%%%%%%%%%%%%%%%%%%%%%%%%%%
\begin{proof}
	
 Suppose  $\nu$ is an isomorphism and $\ker\mu \cap {^M[G_1,G_1]}=1$.
  Since $\mu'$ is a restriction of $\mu$ on ${^M[G_1,G_1]}$, $\ker\mu'=\ker\mu\cap {^M[G_1,G_1]}=1$. This implies $\mu'$ is one-one. For surjectivity, without any loss, we may take $y=(g_1'\star g_2')[g_3',g_4']\in{^M[G_2,G_2]}$ for some $g_1',g_2',g_3',g_4'\in G_2$. Then $\beta_2(g_i')=\nu(\beta_1(g_i))$ for some $g_i \in G_1$, where $i=1,2,3,4$. Since $(\lambda,\mu,\nu)$ is a morphism, we have $\beta_2(g_i')=\beta_2(\mu(g_i))$. This implies $g_i'\mu(g_i)^{-1} \in \ker \beta_2 \subseteq \mathcal{Z}(G_2)$.  Therefore, we have $(g_1\star g_2)[g_3,g_4] \in {^M[G_1,G_1]}$ such that
			\begin{align*}
				\mu'((g_1\star g_2)[g_3,g_4])  & = (\mu(g_1)\star \mu(g_2))[\mu(g_3),\mu(g_4)]\\
				& = (g_1'\star g_2')[g_3',g_4']=y.
			\end{align*}
			Hence, $\mu'$ is a multiplicative Lie algebra isomorphism. It is easy to see that the following diagram is commutative:
			\[
			\xymatrix{
				^M[G_1,G_1] \ar[d]^{\mu'} & \ar[l]_{\delta_1^c} K_1 \times K_1 \ar[r]^{\delta_1^s} \ar[d]^{\nu \times \nu} & ^M[G_1,G_1] \ar[d]_{\mu'} \\
				^M[G_2,G_2] & \ar[l]_{\delta_2^c}  K_2 \times K_2 \ar[r]^{\delta_2^s} & ^M[G_2,G_2],
			}
			\]
			where $\delta_i^c(k_{i1},k_{i2})=[g_{i1},g_{i2}]$ and $\delta_i^s(k_{i1}',k_{i2}')=(g_{i1}'\star g_{i2}')$ with $\beta_i(g_{ij})=k_{ij}$ and $\beta_i(g_{ij}')=k_{ij}'$, where $ i,j \in \{1,2\} $. The converse is obvious.
	\end{proof}	
 
 %%%%%%%%%%%%%%%%%%%%%%%%%%%%%%%%%%%%%%%%%%%%%%%%%%%%%%%%%%%%%%%%
 \begin{corollary}
		Let $(\lambda,\mu,\nu)$ be an isoclinic morphism between the extensions $\mathcal{E}_1$ and $\mathcal{E}_2$. Then $\ker \mu \cap{^M[G_1,G_1]}=1$ and $(\im \mu )H_2 =G_2$.
  \end{corollary}
  %%%%%%%%%%%%%%%%%%%%%%%%%%%%%%%%%%%%%%%%%%%%%%%%%%%%%%%%%%%%%%
  
		\begin{proof}
			Let $g' \in G_2$. There exists $g\in G_1$ such that $g'\mu(g)^{-1} \in \ker \beta_2 = H_2$. Hence, the result follows.
		\end{proof}
%%%%%%%%%%%%%%%%%%%%%%%%%%%%%%%%%%%%%%%%%%%%%%%%%%%%%%%%%%%%%%%%
In particular,
 \begin{corollary}
		Let $(\lambda,\mu,\nu)$ be an isoclinic morphism between the extensions $\mathcal{E}_1$ and $\mathcal{E}_2$ such that $H_1=\mathcal{Z}(G_1)$ and $H_2=\mathcal{Z}(G_2)$. Then $\ker \mu \cap{^M[G_1,G_1]}=1$ and $(\im \mu )\mathcal{Z}(G_2) =G_2$.
  \end{corollary}
%%%%%%%%%%%%%%%%%%%%%%%%%%%%%%%%%%%%%%%%%%%%%%%%%%%%%%%%%%%%%%%%
         Conversely,
		\begin{proposition}\label{second_proposition}
		Let $\mu$ be a multiplicative Lie algebra homomorphism from $G_1$ to $G_2$ such that $\ker \mu \cap{^M[G_1,G_1]}=1$,  $\im \mu H_2=G_2$, $H_1=\mathcal{Z}(G_1)$, and $H_2=\mathcal{Z}(G_2)$. Then a morphism $(\lambda,\mu,\nu)$ is an isoclinic morphism between $\mathcal{E}_1$ and $\mathcal{E}_2$.
	\end{proposition}
 %%%%%%%%%%%%%%%%%%%%%%%%%%%%%%%%%%%%%%%%%%%%%%%%%%%%%%%%%%%%%%%%
 
	\begin{proof}
		To prove our proposition, it is enough to show that $\nu$ is a bijection.  The rest follows from Theorem \ref{first_theorem}. Let $k'\in K_2$. Then, there exists  $g'\in G_2$  such that $\beta_2(g') = k'$. But $\im \mu H_2=G_2$ implies that  $k' = \beta_2(\mu(g)h') = \beta_2(\mu(g)) = \nu(\beta_1(g))$ for some $g\in G_1$ and $h'\in H_2$. Thus, $\nu$ is surjective. Now, suppose $k_1\in \ker\nu$. Then, we have  $\beta_2(\mu(g_1))=1$ for some $g_1\in G_1$. This implies that $\mu(g \star g_1)=\mu(g)\star \mu(g_1)=1$ and $\mu([g,g_1])=[\mu(g),\mu(g_1)]=1$ for any $g\in G_1$. Hence $(g\star g_1)$, $[g,g_1]\in \ker\mu \cap{^M[G_1,G_1]}=1$. Therefore $g_1\in \mathcal{Z}(G_1)$ and so $\nu$ is injective. 
	\end{proof}
 %%%%%%%%%%%%%%%%%%%%%%%%%%%%%%%%%%%%%%%%%%%%%%%%%%%%%%%%%%%%%%%%
        The following corollary is an immediate consequence of  Proposition \ref{second_proposition}.
        \begin{corollary}
            Let $\mu$ be a mutiplicative Lie algebra homomorphism from $G_1$ to $G_2$ such that $\ker \mu \cap{^M[G_1,G_1]}=1$ and $(\im \mu) \mathcal{Z}(G_2)=G_2$. Then the pair $(\nu,\mu|_{^M[G_1,G_1]})$ is an isoclinism between $G_1$ and $G_2$.
        \end{corollary}
  %%%%%%%%%%%%%%%%%%%%%%%%%%%%%%%%%%%%%%%%%%%%%%%%%%%%%%%%%%%%
  
	\begin{lemma}\label{second_lemma}
		Let $\mathcal{C}_j \equiv \xymatrix{1\ar[r] & H_j\ar[r]^{i_j} & G_j \ar[r]^{\beta_j} & K_j\ar[r] & 1,}$ $j=1,2$, be two central extensions of multiplicative Lie algebras. Then 
		\begin{enumerate}
			\item $\mathcal{C}_1$ is a stem extension if and only if $I \cap { ^M[G_1,G_1]}\neq 1$ for every non-trivial ideal $I$ of $H_1$.
   
			\item If $\mathcal{C}_1$ and $\mathcal{C}_2$ are two isoclinic stem extensions, then $H_1$ is isomorphic to $H_2$. 
		\end{enumerate}
\end{lemma}
%%%%%%%%%%%%%%%%%%%%%%%%%%%%%%%%%%%%%%%%%%%%%%%%%%%%%%%%%%%%%%%%%%%%%%%%
\begin{proof}
			\textit{(1)~} Let $\mathcal{C}_1$ be a stem extension and $I$ be a non-trivial ideal of $H_1$. On the contrary, suppose $I \cap { ^M[G_1,G_1]}=1$. Then $I=I\cap H_1 \subseteq I \cap {^M[G_1,G_1]}=1$, which is a contradiction.
   
   Conversely, assume  $I \cap { ^M[G_1,G_1]}\neq 1$ for every non-trivial ideal $I$ of $H_1$. Suppose  $\mathcal{C}_1$ is not a stem extension.
   Then, there exists $h\in H_1$ which is not in $^M[G_1,G_1]$. If we take an ideal $I =\textless h \textgreater$ of $H_1$, then  $I$ is a non-trivial ideal of $H_1$ such that $I \cap {^M[G_1,G_1]}=1$, which is not possible. 
   
	\textit{(2)~} Let  $\mathcal{C}_1$ and $\mathcal{C}_2$ be two stem extensions such that $\mathcal{C}_1 \sim_{ml} \mathcal{C}_2$ via $(\lambda,\mu)$. From Lemma \ref{first_lemma}, we have $\mu(H_1\cap {^M[G_1,G_1]})=H_2 \cap {^M[G_2,G_2]}$, which implies $\mu(H_1)=H_2$. Hence, $H_1$ is isomorphic to $H_2$.
	\end{proof}
%%%%%%%%%%%%%%%%%%%%%%%%%%%%%%%%%%%%%%%%%%%%%%%%%%%%%%%%%%%%%%%%%%%%%%%%
The following corollary is due to Definition \ref{D2} and Lemma \ref{second_lemma}.
\begin{corollary}
   Let  $\mathcal{C}_1$ and $\mathcal{C}_2$ be two stem covers such that $\mathcal{C}_1 \sim_{ml} \mathcal{C}_2$ via $(\lambda,\mu)$. Then, the Schur multipliers of $K_1$ and $K_2$ are isomorphic, so the corresponding covers are also isomorphic.
\end{corollary}
%%%%%%%%%%%%%%%%%%%%%%%%%%%%%%%%%%%%%%%%%%%%%%%%%%%%%%%%%%%%%%%%%%%%%%%%%%

\begin{proposition}
		Every central extension of multiplicative Lie algebras is isoclinic to a stem extension.
\end{proposition}
%%%%%%%%%%%%%%%%%%%%%%%%%%%%%%%%%%%%%%%%%%%%%%%%
\begin{proof}
			Let $\mathcal{E} \equiv \xymatrix{1\ar[r] & H \ar[r]^{i} & G \ar[r]^{\beta} & K \ar[r] & 1}$ be a central extension. Consider $\mathcal{I}=\{I: I \subseteq H ~ \text{and}~ I\cap{ ^M[G,G]}=1\}$, then $\mathcal{I}$ is non empty as $1\in \mathcal{I}$. By Zorn's Lemma $\mathcal{I}$ has a maximal element, say $J$ (with respect to inclusion). Clearly,
			\[{\mathcal{E}_J} \equiv \xymatrix{1\ar[r] & H/J \ar[r]^{i} & G/J \ar[r]^{\bar{\beta}} & K \ar[r] & 1},\] 
			where $\bar{\beta}(gJ)=\beta(g)$, is also a central extension. Then $\mathcal{E} \sim_{ml} {\mathcal{E}_J}$ via $(I_K,\mu)$, where $\mu :{^M\left[\frac{G}{J},\frac{G}{J}\right]}\to {^M[G,G]}$ defined by $\mu((g_1J\star g_2 J)[g_3J, g_4 J])=(g_1\star g_2)[g_3,g_4]$.  Now, using Lemma \ref{second_lemma}, we prove that ${\mathcal{E}_J}$ is a stem extension. Assume that $\frac{I}{J} \cap {^M\left[\frac{G}{J},\frac{G}{J}\right]}=J$ for some non-trivial ideal $I$ of $H$ such that $J \subseteq I \subseteq H$. If $x\in I \cap {^M[G,G]}$, then $xJ \in \frac{I}{J} \cap {^M\left[\frac{G}{J},\frac{G}{J}\right]}=J$. Thus $x\in J$, $x\in J \cap { ^M[G,G]}=1$. Hence $I \cap {^M[G,G]}=1$, this implies $I\in \mathcal{I}$ and $I=J$.
	\end{proof}
%%%%%%%%%%%%%%%%%%%%%%%%%%%%%%%%%%%%%%%%%%%%%%%%%%%%%%%%%

Let $$\mathcal{F} \equiv\xymatrix{1\ar[r] & R \ar[r]^{i} & F \ar[r]^{\gamma} & K \ar[r] & 1}$$ be a free presentation of a multiplicative Lie algebra $K$. Recall from \cite{RLS}, the Schur multiplier of $K$ is given by $\tilde M(K)= \frac{^M[F,F]\cap R}{^M[R,F]}$. The Schur multiplier $\tilde{M}$ defines a covariant functor from the category $ML$ of multiplicative Lie algebras to the category of abelian groups \cite[Corollary 6.9]{RLS}. Thus a homomorphism $\eta$ from $G$ to $K$ induces a homomorphism $\tilde{M}(\eta): \tilde{M}(G) \to \tilde{M}(K)$. \\

The following lemma is a consequence of \cite[Proposition 4.20]{AMS} and is useful for further investigation.
\begin{lemma}\label{third_lemma}
	Let $K$ be a multiplicative Lie algebra. Then there exist a homomorphism $\Delta : \tilde{M}\left(\frac{K}{\mathcal{Z}(K)}\right)\to{ ^M[K,K]}$ such that the  sequence 
	$$\xymatrix{1\ar[r] & \ker \Delta \ar[r] & \tilde{M}\big(\frac{K}{\mathcal{Z}(K)}\big) \ar[r]^{\Delta} & {^M[K,K]} \ar[r]^{{p'}~~~} & {^M\big[\frac{K}{\mathcal{Z}(K)},\frac{K}{\mathcal{Z}(K)}\big]} \ar[r] & 1}$$
	is exact, where $p'=p|_{^M[K,K]}$ and $p: K \to \frac{K}{\mathcal{Z}(K)}$.
\end{lemma}
%%%%%%%%%%%%%%%%%%%%%%%%%%%%%%%%%%%%%%%%%%%%%%%%%%%%%%%%%%%%%%%%%%%

\begin{proof}
     Let $\mathcal{F} \equiv\xymatrix{1\ar[r] & R \ar[r]^{i} & F \ar[r]^{\gamma} & K \ar[r] & 1}$ be a free presentation of $K$. Then $\mathcal{Z}(K)\cong \frac{T}{R}$ for some ideal $T$ of $F$. Since ${^M[\mathcal{Z}(K),K]}=1$, we get $^M[T,F] \subseteq R$. Define a map $\delta : {^M[F,F]}\cap T\to \frac{^M[F,F]}{^M[F,F]\cap R}$, given by $\delta(x)=x{(^M[F,F]\cap R)}$. Then $\delta$ is a well-defined homomorphism such that $\ker(\delta)\cong {^M[F,F]}\cap R$. Thus, we have an induced homomorphism $\Delta:\tilde{M}\big(\frac{K}{\mathcal{Z}(K)}\big)\to {^M[K,K]}$ such that $\im \Delta \cong {^M[K,K]}\cap \mathcal{Z}(K)$. Also, $\ker p'={^M[K,K]}\cap \mathcal{Z}(K)$. Thus the given sequence is exact.
\end{proof}
%%%%%%%%%%%%%%%%%%%%%%%%%%%%%%%%%%%%%%%%%%%%%%%%%%%%%%%%%%%%%%%%%%%

\begin{remark}\label{remark_third}
	Let $\mathcal{E} \equiv \xymatrix{1\ar[r] & H \ar[r]^{i} & G \ar[r]^{\beta} & K \ar[r] & 1}$ be a central extension of $H$ by $K$. Then, there exist a homomorphism $\Delta:\tilde{M}(K) \to {^M[G,G]}$ such that the sequence 
	$$\xymatrix{1\ar[r] & \ker \Delta \ar[r] & \tilde{M}(K) \ar[r]^{\Delta~} & ^M[G,G] \ar[r]^{\beta'~} & ~^M[K,K] \ar[r] & 1}$$  is exact, where $\beta'=\beta|_{^M[G,G]}$.
\end{remark}
%%%%%%%%%%%%%%%%%%%%%%%%%%%%%%%%%%%%%%%%%%%%%%%%%%%%%%%%%%%%%%%%%%%
The following lemma gives another necessary and sufficient condition for a morphism to be an isoclinic morphism using Theorem \ref{first_theorem}.
\begin{lemma}\label{fourth_lemma}
	Let $(\lambda,\mu,\nu)$ be a morphism between central extensions $\mathcal{E}_1$ and $\mathcal{E}_2$ such that $\nu$ is an isomorphism. Then $(\lambda,\mu,\nu)$ is an isoclinic morphism if and only if $\tilde{M}(\nu)(\ker \Delta_1)=\ker\Delta_2$.
\end{lemma}
%%%%%%%%%%%%%%%%%%%%%%%%%%%%%%%%%%%%%%%%%%%%%%%%%%%%%%%%%%%%%

\begin{proof}
	Using Remark \ref{remark_third} and morphism $(\lambda,\mu,\nu)$, we have the following commutative diagram
		\[\xymatrix{
	 1 \ar[r] & \ker \Delta_1 \ar[r] \ar[d]^{\tilde{M}(\nu)|_{\ker\Delta_1}} & \tilde{M}(K_1) \ar[r]^{\Delta_1~} \ar[d]^{\tilde{M}(\nu)} & ^M[G_1,G_1] \ar[r]^{\beta_1} \ar[d]^{\mu|_{^M[G_1,G_1]}} & ^M[K_1,K_1] \ar[r] \ar[d]^{\nu|_{^M[K_1,K_1]}}  & 1 \\
		 1 \ar[r] & \ker \Delta_2 \ar[r] & \tilde{M}(K_2) \ar[r]^{\Delta_2~} & ^M[G_2,G_2] \ar[r]^{\beta_2} & ^M[K_2,K_2] \ar[r] & 1.
	}\]
Let $(\lambda,\mu,\nu)$ be an isoclinic morphism. Then $\nu$ and  $\mu|_{^M[G_1,G_1]}$ are isomorphisms. Now, suppose $k_2 \in \ker \Delta_2$. Then, there exists $k_1\in \tilde{M}(K_1)$ such that $\tilde{M}(\nu)(k_1)=k_2$. This implies $1=\Delta_2(k_2)=\Delta_2(\tilde{M}(\nu)(k_1))=\mu|_{^M[G_1,G_1]}(\Delta_1(k_1))$ and so $k_1\in \ker\Delta_1$. Hence $\tilde{M}(\nu)(\ker\Delta_1)=\ker \Delta_2$.

Conversely, suppose $\nu$ is an isomorphism and $\tilde{M}(\nu)(\ker\Delta_1)=\ker \Delta_2$. So,  $\tilde{M}(\nu)$ and $\nu|_{^M[K_1,K_1]}$ are isomorphisms. Thus, $\tilde{M}(\nu)|_{\ker\Delta_1}$ is also an isomorphism. Hence, we  conclude that $\mu|_{^M[G_1,G_1]}$ is an isomorphism. Now, Theorem \ref{first_theorem} implies that  $(\lambda,\mu,\nu)$ is an isoclinic morphism. 
	
\end{proof}
%%%%%%%%%%%%%%%%%%%%%%%%%%%%%%%%%%%%%%%%%%%%%%%%%%%%%%%%%%%%%%%%

 Let $\mathcal{E}_i \equiv \xymatrix{1\ar[r] & H_i\ar[r]^{i} & G_i \ar[r]^{\beta_i} & K_i\ar[r] & 1,}$ $i=1,2$, be two central extensions and $\nu : K_1 \to K_2$ be a homomorphism. Form the pullback $$G_1\times_{K_2} G_2=\{(g_1,g_2): g_i\in G_i, i=1,2,~ \nu(\beta_1(g_1))=\beta_2(g_2)\}.$$ Then $G_1\times_{K_2} G_2$ is a subalgebra of $G_1\times G_2$ and the extension 
\[\bar{\mathcal{E}} \equiv \xymatrix{1\ar[r] & H_1\times H_2 \ar[r]^{i~~} & G_1\times_{K_2} G_2 \ar[r]^{~\beta} & K_1 \ar[r] & 1}\]
where $\beta(g_1,g_2)=\beta_1(g_1)$ for all $(g_1,g_2)\in G_1\times_{K_2} G_2$, is a central extension.
\begin{remark}
   In further investigation, we use the notation  $\bar{G}$ instead of $G_1\times_{K_2} G_2$.
\end{remark}

%%%%%%%%%%%%%%%%%%%%%%%%%%%%%%%%%%%%%%%%%%%%%%%%%%%%%%%%%%

\begin{lemma}\label{fifth_lemma}
	Let $\nu: K_1 \to K_2$ be an isomorphism and $\bar{\mathcal{E}}$ be a central extension defined as above. Then $\ker \Delta = \tilde{M}(\nu)^{-1}(\ker \Delta_2)\cap \ker \Delta_1$, where $\Delta: \tilde{M}(K_1)\to {^M[\bar{G},\bar{G}]}$.
\end{lemma}
%%%%%%%%%%%%%%%%%%%%%%%%%%%%%%%%%%%%%%%%%%%%%%%%%%%%%%%%%%
\begin{proof}
	Define natural projections $p_i:\bar{G} \to G_i$ by $p_i(g_1,g_2)=g_i$ for $i=1,2$. Let $\nu_1=I$ and $\nu_2=\nu$. Then $(p_i|_{H_1\times H_2},p_i, \nu_i)$ for $i=1,2$, are morphisms from $\bar{\mathcal{E}}$ to $\mathcal{E}_1$ and $\bar{\mathcal{E}}$ to $\mathcal{E}_2$, respectively. Thus, we have two commutative diagrams similar to the previous lemma. This implies $p_1(\Delta(x)) =\Delta_1(x)$ and $p_2(\Delta(x))=\Delta_2(\tilde{M}(\nu)(x))$ for all $x\in \tilde{M}(K_1)$. Therefore, $\Delta(x) = (\Delta_1(x)),\Delta_2(\tilde{M}(\nu)(x))$. Hence, the result follows.
\end{proof}

%%%%%%%%%%%%%%%%%%%%%%%%%%%%%%%%%%%%%%%%%%%%%%%%%%%%%%%%%%
\begin{theorem}\label{second_theorem}
	Let $\mathcal{E}_i \equiv \xymatrix{1 \ar[r] & H_i \ar[r]^{i} & G_i \ar[r]^{\beta_i} & K_i \ar[r] & 1,}$ $i=1,2$ be two central extensions and $\nu: K_1 \to K_2$ be an isomorphism. Then the following are equivalent:
	\begin{enumerate}
		\item $\mathcal{E}_1$ is isoclinic to $\mathcal{E}_2$.
		\item There exists an isomorphism $\mu':{^M[G_1,G_1]} \to {^M[G_2,G_2]}$ such that $\mu'\Delta_1=\Delta_2 \tilde{M}(\nu)$.
		\item $\tilde{M}(\nu)(\ker \Delta_1)=\ker \Delta_2$.
	\end{enumerate}
\end{theorem}
%%%%%%%%%%%%%%%%%%%%%%%%%%%%%%%%%%%%%%%%%%%%%%%%%%%%%%%%%%%%
\begin{proof}
	\textit{(1)} $\implies$ \textit{(2)} Let $\mathcal{E}_1 \sim_{ml} \mathcal{E}_2$ via $(\nu, \mu' )$. Clearly $^M[\bar{G},\bar{G}]=\{(g_1,\mu'(g_1)):g_1\in {^M[G_1,G_1]}\}$. Let $\mu_i: \bar{G} \to G_i$, $i=1,2$ be defined as $\mu_i(g_1,g_2)=g_i$ and $\nu_1=I$ and $\nu_2=\nu$. Then $\ker \mu_i\cap {^M[\bar{G},\bar{G}]}=1$. Thus $(\mu_i|_{H_1\times H_2},\mu_i,\nu_i)$, $i=1,2$ are isoclinic morphisms from $\bar{\mathcal{E}}$ to $\mathcal{E}_1$ and $\bar{\mathcal{E}}$ to $\mathcal{E}_2$, respectively. Therefore $\mu_i|_{^M[\bar{G},\bar{G}]}\Delta= \Delta_i\tilde{M}(\nu_i)$, $i=1,2$. Now take $\mu'=(\mu_2|_{^M[\bar{G},\bar{G}]})(\mu_1|_{^M[\bar{G},\bar{G}]})^{-1}$. Then $\mu' \Delta_1=\Delta_2\tilde{M}(\nu)$.

 	\textit{(2)} $\implies$ \textit{(3)}  
		Follows from the Lemma \ref{fourth_lemma}.
		
	\textit{(3)} $\implies$ \textit{(1)}  Let $\tilde{M}(\nu)\ker \Delta_1=\ker \Delta_2$. By Lemma \ref{fifth_lemma}, $\ker \Delta=\ker \Delta_1 $. Thus $\tilde{M}(\nu_1)\ker \Delta = \ker \Delta_1$ and $\tilde{M}(\nu_2) \ker \Delta=\ker \Delta_2$, where $\nu_1=I$ and $\nu_2=\nu$. Hence by Lemma \ref{fourth_lemma}, there exist isoclinic morphisms from $\bar{\mathcal{E}}$ to $\mathcal{E}_1$ and $\bar{\mathcal{E}}$ to $\mathcal{E}_2$. Let $(\nu_i,\mu_i')$ be isoclinisms from $\bar{\mathcal{E}}$ to $\mathcal{E}_1$ and $\bar{\mathcal{E}}$ to  $\mathcal{E}_2$. Then $\mathcal{E}_1$ is isoclinic to $\mathcal{E}_2$ via $(\nu,\mu')$, where $\mu'= \mu_2'\circ \mu_1'^{-1}$.\end{proof}

%%%%%%%%%%%%%%%%%%%%%%%%%%%%%%%%%%%%%%%%%%%%%%%%%%%%%%%%%%%%%%
It is well known that the cover of a multiplicative Lie algebra need not be unique \cite[Remark 4.15]{AMS}. In the case of a perfect multiplicative Lie algebra, the covers are uniquely determined up to isomorphism \cite[Corollary 4.11]{AMS}. The following corollary shows that the covers of a general multiplicative Lie algebra can be determined up to isoclinism.

\begin{corollary}
    All stem covers of a multiplicative Lie algebra $K$ are mutually isoclinic.
\end{corollary}
\begin{proof}
 Let $\mathcal{C}_i \equiv \xymatrix{1 \ar[r] & H_i \ar[r]^{i} & G_i \ar[r]^{\beta_i} & K \ar[r] & 1}$ $i=1,2$ be two stem covers of a  multiplicative Lie algebra $K$ and $\nu=I$ be an isomorphism on $K$. Now, the result follows from Theorem \ref{second_theorem}.  

 In particular, covers of a multiplicative Lie algebra $K$ are mutually isoclinic.
\end{proof}

%%%%%%%%%%%%%%%%%%%%%%%%%%%%%%%%%%%%%%%%%%%%%%%%%%%%%%%%%%%%%%%%%%%%%

\noindent{\bf Acknowledgement:}
The first named author sincerely thanks MNNIT Allahabad and the Ministry of Education, Govt. of India, for providing the institute fellowship. 
The second named author sincerely thanks IIIT Allahabad and the University Grant Commission (UGC), Govt. of India, New Delhi, for the research fellowship. The third named author is thankful to National Board for Higher Mathematics (NBHM), Government of India  for the financial support for the project ``Linear Representation of multiplicative Lie algebra".
%%%%%%%%%%%%%%%%%%%%%%%%%%%%%%%%%%%%%%%%%%%%%%%%%%%%%%%%%%%%%%%%%%%%%%%%

\end{document}